\documentclass{article}
\usepackage{euscript,amsmath,amssymb,amsfonts,graphicx,bm,amsthm,mathtools}
  \usepackage{caption}
\usepackage{subcaption}
\usepackage{pdflscape}

\usepackage{array}
\usepackage{bbm}

\usepackage{cite} % this makes it so numbered citation list like [1,2,3,4,5,6] would be [1-6].

  \usepackage{paralist}
  \usepackage{graphics} 
\usepackage{soul}

\usepackage{latexsym,amssymb,enumerate,amsmath,amsthm} 
  \usepackage{graphicx} 
  \usepackage{tikz}
     \usepackage[colorlinks=false]{hyperref}
 \hypersetup{urlcolor=blue, citecolor=red}
\usepackage{latexsym,amssymb,enumerate,amsmath} 
\usepackage{pgfplots}
\pgfplotsset{compat=1.18}
\usepackage{soul,xcolor}

\newtheorem{theorem}{Theorem}

\newtheorem{lemma}{Lemma}
\newtheorem{corollary}{Corollary}
\theoremstyle{plain}
\theoremstyle{remark}

\theoremstyle{definition}
\newcommand{\dd}{\textup{d}}
\def\eps{\varepsilon}
\def\E{\mathbb{E}}
\def\P{\mathbb{P}}
\def\R{\mathbb{R}}

\begin{document}

% Use the \preprint command to place your local institutional report
% number in the upper righthand corner of the title page in preprint mode.
% Multiple \preprint commands are allowed.
% Use the 'preprintnumbers' class option to override journal defaults
% to display numbers if necessary
%\preprint{}

%Title of paper
\title{First passage times with fast immigration}

% repeat the \author .. \affiliation  etc. as needed
% \email, \thanks, \homepage, \altaffiliation all apply to the current
% author. Explanatory text should go in the []'s, actual e-mail
% address or url should go in the {}'s for \email and \homepage.
% Please use the appropriate macro foreach each type of information

% \affiliation command applies to all authors since the last
% \affiliation command. The \affiliation command should follow the
% other information
% \affiliation can be followed by \email, \homepage, \thanks as well.
\author{Hwai-Ray Tung\thanks{Department of Mathematics, University of Utah, Salt Lake City, UT 84112 USA (\texttt{ray.tung@utah.edu}).} \and Sean D. Lawley\thanks{Department of Mathematics, University of Utah, Salt Lake City, UT 84112 USA (\texttt{lawley@math.utah.edu}).}
}
\date{\today}
\maketitle

%\newpage
\begin{abstract}
Many scientific questions can be framed as asking for a first passage time (FPT), which generically describes the time it takes a random ``searcher'' to find a ``target.'' The important timescale in a variety of biophysical systems is the time it takes the fastest searcher(s) to find a target out of many searchers. Previous work on such fastest FPTs assumes that all searchers are initially present in the domain, which makes the problem amenable to extreme value theory. In this paper, we consider an alternative model in which searchers progressively enter the domain at a constant ``immigration'' rate. In the fast immigration rate limit, we determine the probability distribution and moments of the $k$-th fastest FPT. Our rigorous theory applies to many models of stochastic motion, including random walks on discrete networks and diffusion on continuous state spaces. Mathematically, our analysis involves studying the extrema of an infinite sequence of random variables which are both not independent and not identically distributed. Our results constitute a rare instance in which extreme value statistics can be determined exactly for strongly correlated random variables.
\end{abstract}

%%%%%%%%%%%%%%%%%%%%%%%%%%%%%%%%%%%%%%%%%%%%%%%%%%%%%%%%%%%%%%%%%%%%%%%%%%%%%%%%%%%%%%%%%%%%%%%%%%%%%%%%%%%%%%%%%%%%%%%%%%%%%%%%%%%%%%%%%%%%

\section{Introduction}

The timing of many processes in biology, chemistry, physics, and other fields can be understood in terms of first passage times (FPTs) \cite{redner2001}. Generically, a FPT is the time that it takes a random ``searcher'' to find a ``target.'' Depending on the application, the searcher and target could be, for example, a ligand and a receptor \cite{shoup82}, a sperm and an egg \cite{meerson2015}, or a predator and a prey \cite{kurella2015}. The FPT need not describe search in physical space, but could for instance be the time it takes a cell lineage to acquire a target number of mutations and thereby become cancerous \cite{armitage1954age, nordling1953new, durrett2015branching}.

The important timescale in many systems is the time it takes the fastest searcher to find a target out of many searchers, which is called the extreme FPT or fastest FPT \cite{lawley2024competition}. Extreme FPTs have been studied in the physics literature for decades \cite{weiss1983, yuste1996, yuste1997escape, yuste2000, yuste2001, redner2014} and have been more recently considered in biological contexts \cite{schuss2019, lawley2020dist, morgan2023modulation, bernoff2023single, maclaurin2025extreme}. A prototypical example is human fertilization, which is triggered by the fastest sperm cell to find the egg out of roughly $10^8$ sperm cells \cite{meerson2015}. Extreme FPTs have been modeled by assuming that there are initially $N\gg1$ searchers with independent and identically distributed (iid) FPTs $\tau_1,\dots,\tau_N$ (i.e.\ the times that it would take each individual searcher to find the target). The extreme FPT is then given by the minimum of $\tau_1,\dots,\tau_N$. Mathematically, this formulation is convenient because it allows the use of extreme value theory \cite{fisher1928, colesbook, haanbook, lawley2020dist}, which is a century-old theory for estimating the minimum of $N\gg1$ iid random variables. Importantly, extreme value theory generally requires the random variables to be iid, and indeed, very little is known in the non-iid case \cite{majumdar2020}.

In this paper, we suppose that there is initially only a single searcher, but that additional searchers enter the system (or ``immigrate'' or are ``born'') at a constant rate. We study the time it takes the fastest (or $k$-th fastest) searcher out of this growing population of searchers to find a target. This latter model is natural in applications in which the searchers are not all initially present, but rather enter the system progressively. For example, Campos and M\'endez \cite{campos2024} recently introduced this model to study foraging by ants and other eusocial species. Another example is the time for a person to be infected with a virus when seated near an infected person constantly exhaling viruses. From a mathematical standpoint, this problem formulation presents new challenges since extreme value theory no longer applies. Indeed, we show below that the fastest search time in this model is the minimum of an infinite sequence of random variables which are not independent and not identically distributed.

In the fast immigration rate limit, we obtain the full probability distribution and all the moments of the $k$-th fastest FPT based on the short-time behavior of the probability distribution of the FPT of a single searcher. Our results cover common models of stochastic search, including where each searcher is (i) a continuous-time Markov jump process on an arbitrary discrete network, or (ii) a diffusion process on a continuous state space in arbitrary dimension. We illustrate our results with numerical simulations in a few canonical search models.

The paper is organized as follows. Section~\ref{sec:model} formulates the model precisely. Section~\ref{sec:results} gives the limiting distributions for fast immigration and moments of the $k$-th fastest FPT. Section~\ref{sec:numerics} compares the analytical theory to numerical computations. Section~\ref{sec:discussion} discusses relations to prior work. An Appendix collects the mathematical proofs.

%%%%%%%%%%%%%%%%%%%%%%%%%%%%%%%%%%%%%%%%%%%%%%%%%%%%%%%%%%%%%%%%%%%%%%%%%%%%%%%%%%%%%%%%%%%%%%%%%%%%%%%%%%%%%%%%%%%%%%%%%%%%%%%%%%%%%%%%%%%%%%%%%%%%%%%%%%%%%%%%%%%%%%%%%%%%%%%%%%%%%%%%%%%%%%%%%%%%%%%%%%%%%%%%%%%%%%%%%%%%%%%%%%%%%%%%%%%%%%%%
\section{Model}
\label{sec:model}

We start with one searcher whose initial spatial position is sampled from a given probability distribution (which could be a Dirac delta mass at a single point). This initial searcher then randomly explores the spatial domain until it reaches some given target region(s). At rate $\lambda>0$, new searchers immigrate into the system and move randomly until hitting the target. The initial position and stochastic motion of each searcher are independent of other searchers and follow the same probability laws. We are interested in the time that it takes the fastest $k\ge1$ searchers to find the target, which we denote by $T_k$.

Mathematically, we ultimately need only specify the probability distribution of the time it takes a single searcher to find the target after it has entered the system. We let $\tau$ denote this random time and describe its probability distribution by its survival probability,
\begin{align*}
    S(t)
    =P(\tau>t).
\end{align*}
We emphasize that all of the information about the spatial search process (initial searcher position, spatial domain, searcher motion, target location(s), etc.)\ is encapsulated in the function $S(t)$.

The collection of search times is then given by
\begin{align}\label{eq:collection}
    \Big\{\tau_1,\tau_2+\frac{1}{\lambda}\sigma_1,\tau_3+\frac{1}{\lambda}(\sigma_1+\sigma_2),\dots\Big\}
    =\Big\{\tau_n+\frac{1}{\lambda}\sum_{i=1}^{n-1}\sigma_i\Big\}_{n\ge1},
\end{align}
where $\{\tau_n\}_{n\ge1}$ is a sequence of iid realizations of $\tau$ and $\{\sigma_i\}_{i\ge1}$ is a sequence of iid unit mean exponential random variables. To explain \eqref{eq:collection} in words, the search time for the initial searcher is $\tau_1$, the search time for the second searcher is the sum of its immigration time, $\sigma_1/\lambda$, and its search time after entering the system, $\tau_2$, the search time for the third searcher is the sum of its immigration time, $(\sigma_1+\sigma_2)/\lambda$, and its search time after entering the system, $\tau_3$, and so on. Notice that the search times $\tau_n+(1/\lambda)\sum_{i=1}^{n-1}\sigma_i$ are neither independent nor identically distributed. The fastest search time is then
\begin{align}\label{eq:T1}
    T_1
    =\min\Big\{\tau_n+\frac{1}{\lambda}\sum_{i=1}^{n-1}\sigma_i\Big\},
\end{align}
and the $k$-th fastest search time is
\begin{align*}
    T_k
    =\min\Big\{\big\{\tau_n+\frac{1}{\lambda}\sum_{i=1}^{n-1}\sigma_i\big\}\Big\backslash\cup_{j=1}^{k-1}T_j\Big\},\quad k\ge1.
\end{align*}

In the fast immigration limit ($\lambda\to\infty$), we determine the limiting distribution of $T_k$ as a function of the short-time behavior of a single searcher, $S(t)=P(\tau>t)$. We consider two classes of short-time behavior of $S(t)$, which encompass a wide variety of processes. The first class is when $1-S(t)=P(\tau\le t)$ (i.e.\ the cumulative distribution function of $\tau$) decays according to a power law at short time,
\begin{align}\label{eq:power}
    1-S(t) \sim At^p, \quad A>0,\, p>0,
\end{align}
where $f\sim g$ is shorthand for $\lim_{t\rightarrow 0^+}f/g =1$. The second class is when $1-S(t)$ decays exponentially at short time,
\begin{align}\label{eq:exponential}
    1-S(t) \sim At^pe^{-C/t}, \quad A>0,\, C>0,\,p\in\R.
\end{align}

An example of a process with the power law decay in \eqref{eq:power} is a searcher on a discrete graph \cite{lawley2020networks}. In this case, $p\ge1$ is the shortest number of jumps from the initial position of the searcher to the target. To find $A$, start by finding all paths from start to target of length of $p$. For each path, multiply the jump rates. Then add the results of each path and divide by $p!$ to get $A$. Another example of search with the power law decay in \eqref{eq:power} is a L{\'e}vy flight, which heuristically features paths that stay in a small neighborhood for a while before making a large jump to another neighborhood. L{\'e}vy flights have been used to describe animal foraging, modeling the idea that an animal lingers in an area to consume available resources, then makes a large movement in search of another resource heavy area \cite{palyulin2014, palyulin2016, tzou2023, gomez2024first, metzler2004}. For L{\'e}vy flights, the power in \eqref{eq:power} is $p=1$ and the prefactor $A$ is a more complicated but known expression \cite{lawley2023super}. Another example of search satisfying \eqref{eq:power} is diffusive search in which searchers can start arbitrarily close to the target, such as searchers who are initially uniformly distributed in the domain \cite{madrid2020comp}.

The exponential decay in \eqref{eq:exponential} describes searchers undergoing diffusion in a wide variety of environments with an initial searcher position that is bounded away from the target(s) \cite{lawley2020uni}, which is a canonical model of stochastic search \cite{redner2001, bressloffbook}. For the case of searchers undergoing $d$-dimensional Brownian motion, the parameter $C>0$ is the diffusion timescale,
\begin{align*}
    C=\frac{L^2}{4D}>0,
\end{align*}
where $L>0$ is the shortest distance from the searcher's initial position(s) to the target(s), and $D$ is the diffusivity of the searcher \cite{lawley2020uni}. The power $p\in\R$ and the prefactor $A>0$ encapsulate finer details of the search process \cite{lawley2020dist}.

%%%%%%%%%%%%%%%%%%%%%%%%%%%%%%%%%%%%%%%%%%%%%%%%%%%%%%%%%%%%%%%%%%%%%%%%%%%%%%%%%%%%%%%%%%%%%%%%%%%%%%%%%%%%%%%%%%%%%%%%%%%%%%%%%%%%%%%%%%%%%%%%%%%%%%%%%%%%%%%%%%%%%%%%%%%%%%%%%%%%%%%%%%%%%%%%%%%%%%%%%%%%%%%%%%%%%%%%%%%%%%%%%%%%%%%%%%%%%%%%%%%%%%%%%%%%%%%%%%%%%%%%%%%%%%%%%%%%%%%%%%%%%%%%%%%%%%%%%%%%%%%%%%%%%%%%%%%%%%%%%%%%%%%%%%%%%%%%%%%%%%%%%%%%%%%%%%%%%%%%
\section{Theoretical results}
\label{sec:results}
In section \ref{sec:SIequality}, we express the survival probability of the faster searcher in the immigration process,
\begin{align*}
    S_I(t)
    =P(T_1>t),
\end{align*}
in terms of the survival probability of a single searcher, $S(t)=P(\tau>t)$. Using this representation, we find in section \ref{sec:limitDist} the limiting distribution of $T_1$ for large $\lambda$. Section \ref{sec:moments} uses the limiting distribution to find the moments of $T_1$. Section \ref{sec:kth} generalizes the limiting distribution and moment results for the fastest FPT, $T_1$, to the $k$-th fastest  FPT, $T_k$.

%%%%%%%%%%%%%%%%%%%%%%%%%%%%%%%%%%%%%%%%%%%%%%%%%%%%%%%%%%%%
\subsection{The survival probability with immigration}
\label{sec:SIequality}
As noted in Campos and Mendez \cite{campos2024} it is possible to express $S_I(t)$ in terms of $S(t)$ and $\lambda$. Rather than their more analytical approach, we give a brief probabilistic derivation. 

Let $S_{I, n}(t)$ denote the survival probability when we condition on having exactly $n$ particles at time $t$. Since the number of particles that have entered the system by time $t$ follows a Poisson distribution with mean $\lambda t$ and we start with 1 particle at time $0$, by the law of total probability,
$$
S_I(t) = \sum_{j=1}^{\infty} \frac{e^{-\lambda t}(\lambda t)^{j-1}}{(j-1)!} S_{I, j}(t).
$$
Next, because the times that new particles immigrate in are generated by a homogeneous Poisson point process, if we condition on the number of immigrants, the arrival times of the immigrants are distributed independently and uniformly on $[0, t]$. Hence,
$$
S_{I, j}(t) = S(t)\left[\frac{1}{t}\int_0^t S(s) ds\right]^{j-1}.
$$
Plugging this back into our equation for $S_I(t)$, we find by the Taylor series of the exponential function that
\begin{equation}
    S_I(t) = S( t) \exp\left[-\lambda \int_0^t 1-S(s) ds \right].
    \label{eqn:SIequality}
\end{equation}

\subsection{Limiting distribution for large $\lambda$}
\label{sec:limitDist}
Classical extreme value theory results like Fisher-Tippett-Gnedenko tell us that extreme values of a large number of iid random variables converges to a Frechet, Weibull, or Gumbel distribution as the number of random variables grows \cite{haanbook, colesbook}. We cannot immediately apply such results since $T_1$ in \eqref{eq:T1} is the minimum of an infinite sequence of random variables which are both not independent and not identically distributed. Even so, we obtain analogous results below. We save their proofs for the Appendix.

\begin{theorem}\label{thm:distpower}
    Let $T_1$ be the FPT of the fastest searcher, where there is immigration at rate $\lambda$. Suppose a single searcher survival probability has the following power law decay,
    \begin{align*}
        1-S(t)
        \sim At^p\quad\text{as }t\to0+,
    \end{align*}
    where $A>0$ and $p>0$. Then we have the following convergence in distribution,
    \begin{align*}
    T_1/a_\lambda\to_\dd\textup{Weibull}(1, p+1),
    \end{align*}
    where
    \begin{align*}
        a_\lambda
        =(A\lambda/(p+1))^{-1/(p+1)}.
    \end{align*}
\end{theorem}
Recall that if a random variable $X$ is $\textup{Weibull}(1, p+1)$, then
\begin{align*}
    P(X> x)
    =\exp(-x^{p+1}),\quad x\in\R.
\end{align*}

\begin{theorem}\label{thm:distexp}
    Let $T_1$ be the FPT of the fastest searcher, where there is immigration at rate $\lambda$. Suppose a single searcher survival probability has the following exponential decay at short time,
    \begin{align}\label{eq:assumptionexp}
        1-S(t)
        \sim At^pe^{-C/t}\quad\text{as }t\to0+,
    \end{align}
    where $C>0$, $A>0$, and $p\in\R$. Then we have the following convergence in distribution,
    \begin{align*}
    \frac{T_1-b_\lambda}{a_\lambda}\to_\dd\textup{Gumbel}(0, 1),
    \end{align*}
    where
    \begin{align*}
        a_\lambda
        =\frac{C}{(\ln (C\lambda))^2},\quad
        b_\lambda
        =\frac{C}{\ln (C\lambda)}
        +\frac{C(p+2)\ln(\ln(C\lambda))}{(\ln (C\lambda))^2}
        -\frac{C\ln(AC^p)}{(\ln (C\lambda))^2}.
    \end{align*}
\end{theorem}
Recall that if a random variable $X$ is $\textup{Gumbel}(0, 1)$, then
\begin{align*}
    P(X> x)
    =\exp(-e^x),\quad x\in\R.
\end{align*}
Note that the choice of $a_\lambda, b_\lambda$ in \eqref{thm:distexp} is not unique. Indeed, it follows from elementary properties of convergence in distribution that $a_\lambda, b_\lambda$ in Theorem~\ref{thm:distexp} can be replaced by any $a_\lambda', b_\lambda'$ satisfying
\begin{align}\label{eq:elem}
    \lim_{\lambda\to\infty}\frac{a_{\lambda}'}{a_{\lambda}}
=1,\quad
\lim_{\lambda\to\infty}\frac{b_{\lambda}'-b_{\lambda}}{a_{\lambda}}
=0.
\end{align}
For more rapid convergence in some cases, 
\begin{corollary}\label{cor:lambert}
    The $a_\lambda , b_\lambda$ pair in Theorem \ref{thm:distexp} can be replaced with
    $$
    a_\lambda = \frac{C}{(p+2)^2W(W+1)}, \qquad b_\lambda = \frac{C}{(p+2)W},
    $$
    where
    \begin{align*}
        W
        =\begin{cases}
            W_0\Big[\frac{1}{p+2}\big(AC^{p+1}\lambda\big)^{\frac{1}{p+2}}\Big] & \text{if }p+2>0,\\
            W_{-1}\Big[\frac{1}{p+2}\big(AC^{p+1}\lambda\big)^{\frac{1}{p+2}}\Big]& \text{if }p+2<0,
        \end{cases}
    \end{align*}
    and $W_0(z)$ denotes the principal branch of the LambertW function and $W_{-1}(z)$ denotes the lower branch \cite{corless1996}.
    \label{coro:lamWparamOfABlam}
\end{corollary}
The proof of Corollary~\ref{cor:lambert} follows from \eqref{eq:elem} and the following asymptotics of the LambertW function \cite{corless1996},
\begin{align*}
    W_0(z)
    &=\ln z-\ln\ln z + o(1)\quad\text{as }z\to\infty,\\
    W_{-1}(z)
    &=\ln(-z)-\ln(-\ln(-z)) + o(1)\quad\text{as }z\to0-.
\end{align*}

The intuition behind these results is that even though the immigration times in \eqref{eq:collection} are dependent, when we condition on the total number of particles at time $t$, we can redefine the $j$-th particle as a randomly chosen particle rather than the $j$-th to immigrate. As noted in section~\ref{sec:SIequality}, this makes the immigration times iid and uniform on $[0, t]$. This would make $T_1$ a minimum of iid random variables. Furthermore, the number of particles is Poisson($\lambda t$) which for large $\lambda$ heuristically will have a number of particles near $\lambda t$. Therefore, by Fisher-Tippett-Gnedenko, we would expect $T_1$ to be roughly Frechet, Weibull, or Gumbel. 

\subsection{Moments for large $\lambda$}
\label{sec:moments}
Instead of the distribution of $T_1$, one may be interested in finding moments like $\E[T_1]$. It turns out we can find the moments using the limiting distributions in Theorems \ref{thm:distpower} and \ref{thm:distexp}. 

In the case where $1-S(t)\sim At^p$, we showed that $S_I$ converges to Weibull$(1, p+1)$, whose $m$-th moment is given in terms of the Gamma function,
$\Gamma(1+\frac{m}{p+1})$. As such, it is natural to conjecture that
\begin{theorem}
\label{thm:momentpower}
    When $1-S(t) \sim At^p$, the moments of the first passage time satisfy
    $$
    \lim_{\lambda\rightarrow \infty} \E\Bigg[\bigg(\bigg(\frac{A\lambda}{p+1}\bigg)^{1/(p+1)}T_1\bigg)^m\Bigg] = \Gamma\bigg(1+\frac{m}{p+1}\bigg).
    $$
\end{theorem}
Hence, if $1-S(t) \sim At^p$, then the mean of $T_1$ decays according to
\begin{align*}
    \E[T_1]
    =\Gamma\bigg(1+\frac{1}{p+1}\bigg)\bigg(\frac{p+1}{A\lambda}\bigg)^{1/(p+1)}
    +o(\lambda^{-1/(p+1)})\quad\text{as }\lambda\to\infty.
\end{align*}

When $1-S(t)\sim At^pe^{-C/t}$, we showed that $S_I$ converges to Gumbel(0, 1), whose $m$-th moment is given in terms of the derivative of the Gamma function, $\frac{d^m}{dt^m}\Gamma(1+t)|_{t=0}$. 
As such, it is natural to conjecture that
\begin{theorem}
\label{thm:momentexp}
    When $1-S(t) \sim At^pe^{-C/t}$, the moments of the first passage time satisfy
    $$
    \lim_{\lambda\rightarrow \infty} \E\left[\left(\frac{T_1-b_\lambda}{a_\lambda}\right)^m\right] = \frac{d^m}{dt^m}\Gamma\left(1+t\right)\bigg\rvert_{t=0}. 
    $$
    Of particular note, when $m=1$ the right hand side is $-\gamma$ and when $m=2$ the right hand side is $\gamma^2 + \pi^2/6$, where $\gamma\approx 0.5772$ is the Euler-Mascheroni constant. 
\end{theorem}
Hence, if $1-S(t) \sim At^pe^{-C/t}$, then Theorem~\ref{thm:distexp} implies that the mean of $T_1$ has the following three-term asymptotic expansion,
\begin{align*}
    \E[T_1]
    &=b_\lambda-\gamma a_\lambda
    +o(a_\lambda)\\
    &=C\Bigg[\frac{1}{\ln (C\lambda)}+\frac{(p+2)\ln(\ln(C\lambda))}{(\ln (C\lambda))^2}-\frac{(\gamma+\ln(AC^p))}{(\ln (C\lambda))^2}\Bigg]+o((\ln\lambda)^{-2})\;\text{as }\lambda\to\infty.
\end{align*}

The proofs of Theorems~\ref{thm:momentpower}-\ref{thm:momentexp} follow from showing uniform integrability of the collection of random variables, $\{((T_1-b_\lambda)/a_\lambda)^m\}_\lambda$, for any given $m$ and are saved for the Appendix.

\subsection{$k$-th fastest searcher}
\label{sec:kth}
We now generalize our results on the time the fastest searcher hits the target, $T_1$, to the time the $k$-th fastest searcher hits the target, $T_k$. Using \eqref{eqn:SIequality} and basic combinatorics, we show
\begin{lemma}
    $T_k$ has survival probability
    \begin{align*}
        &P(T_k > t)\\= &\exp\left[-\lambda \int_0^t 1-S(s)ds\right]\left[S(t)\frac{\left(\lambda \int_0^t 1-S(s)ds\right)^{k-1}}{(k-1)!} + \sum_{j=1}^{k-1} \frac{\left(\lambda \int_0^t 1-S(s)ds\right)^{j-1}}{(j-1)!} \right].
    \end{align*}
    \label{lem:Tksurvival}
\end{lemma}
Using this lemma, we find limiting distributions for $T_k$.
\begin{theorem}
    When $1-S(t)\sim At^p$, then using the $a_\lambda$ from Theorem \ref{thm:distpower},
    \begin{align*}
        \frac{T_k}{a_\lambda}
        \to_\dd Y_k\quad\text{as }\lambda\to\infty,
    \end{align*}
    where the random variable $Y_k\ge0$ has the survival probability 
    \begin{align*}
        \P(Y_k>x)
        =\frac{\Gamma(k,x^{p+1})}{(k-1)!}\quad x\ge0,
    \end{align*}
    where $\Gamma(r,z)$ denotes the upper incomplete gamma function, $\Gamma(r,z)=\int_z^\infty u^{r-1}e^{-u}\,\dd u$.
    \label{thm:kthdistpower}
\end{theorem}

\begin{theorem}
When $1-S(t)\sim At^pe^{-C/t}$, then using the $a_\lambda$ and $b_\lambda$ from Theorem \ref{thm:distexp} (or Corollary~\ref{cor:lambert}),
    \begin{align*}
        \frac{T_k-b_\lambda}{a_\lambda}
        \to_\dd Z_k\quad\text{as }\lambda\to\infty,
    \end{align*}
    where the random variable $Z_k$ has the following probability density function,
    \begin{align*}
        f_{Z_k}(x)
        =\frac{\exp(kx-e^x)}{(k-1)!},\quad x\in\R.
    \end{align*}
    \label{thm:kthdistexp}
\end{theorem}

We can once again show that the moments of $(T_k-b_\lambda)/a_\lambda$ converge to the moments of the limiting distribution, i.e.,

\begin{theorem}
    When $1-S(t)\sim At^p$ and $a_\lambda$ is from Theorem \ref{thm:distpower}, then for any positive integer $m$,
    \begin{align*}
        \lim_{\lambda \rightarrow \infty}\E\left[\left(\frac{T_k}{a_\lambda}\right)^m\right]
        = \E[Y_k^m] = \frac{1}{(k-1)!}\Gamma\left(k+\frac{m}{p+1}\right).
    \end{align*}
    Similarly, when $1-S(t)\sim At^pe^{-C/t}$ and $a_\lambda$ and $b_\lambda$ are from Theorem \ref{thm:distexp} (or Corollary~\ref{cor:lambert}), then for any integer $m\ge1$,
    \begin{align*}
        \lim_{\lambda \rightarrow \infty}\E\left[\left(\frac{T_k-b_\lambda}{a_\lambda}\right)^m\right]
        = \E[Z_k^m] = \frac{1}{(k-1)!}\frac{d^m}{dt^m}\Gamma\left(k+t\right)\bigg\rvert_{t=0}.
    \end{align*}
    \label{thm:kthmoment}
\end{theorem}
Hence, if $1-S(t) \sim At^p$, then the mean of $T_k$ decays according to
\begin{align*}
    \E[T_k]
    =\Gamma\bigg(k+\frac{1}{p+1}\bigg)\bigg(\frac{p+1}{A\lambda}\bigg)^{1/(p+1)}
    +o(\lambda^{-1/(p+1)})\quad\text{as }\lambda\to\infty.
\end{align*}
Moreover, if $1-S(t)\sim At^pe^{-C/t}$, then the mean of $T_k$ has the following three-term asymptotic expansion,
\begin{align*}
    &\E[T_k]
    =b_\lambda+\psi(k) a_\lambda
    +o(a_\lambda)\\
    &\;=C\Bigg[\frac{1}{\ln (C\lambda)}+\frac{(p+2)\ln(\ln(C\lambda))}{(\ln (C\lambda))^2}+\frac{(\psi(k)-\ln(AC^p))}{(\ln (C\lambda))^2}\Bigg]+o((\ln\lambda)^{-2})\;\text{as }\lambda\to\infty,
\end{align*}
where $\psi(k)=\Gamma'(k)/\Gamma(k)=-\gamma+H_{k-1}$ is the digamma function and $H_{k-1}=\sum_{j=1}^{k-1}\frac{1}{j}$ is the $(k-1)$-st harmonic number.
%%%%%%%%%%%%%%%%%%%%%%%%%%%%%%%%%%%%%%%%%%%%%%%%%%%%%%%%%%%%%%%%%%%%%%%%%%%%%%%%
\section{Numerical examples}
\label{sec:numerics}
We now compare the results of section~\ref{sec:results} to numerical computations for three examples of stochastic search; one-dimensional diffusion, three-dimensional diffusion, and a random walk on a discrete grid.

%%%%%%%%%%%%%%%%%%%%%%%%%%%%%%%%%%%%%%%%%%%%%%%%%%%%%%%%%%%%%%%%%%%%%%%%%%%%%%%%%%%%%%%%%%%%%%%%%%%%%%%%%%%%%%%%%%%%%%%%
\subsection{Diffusive search in one space dimension}
The first example is diffusive searchers on the half-line $(0,\infty)$. Specifically, searchers immigrate in to starting position $L>0$ and diffuse with diffusivity $D>0$ until they hit the target located at the origin. The survival probability for a single searcher after entering the system is \cite{carslaw1959}
$$
S(t)
=P(\tau>t)
= \text{erf}\bigg(\sqrt{\frac{L^2}{4Dt}}\bigg),
$$
where erf is the error function. It follows that
$$
1-S(t) \sim At^pe^{-C/t}, \qquad A = \sqrt{\frac{4D}{L^2\pi}}, \qquad p= \frac{1}{2}, \qquad C = \frac{L^2}{4D}.
$$
% where
% \begin{align*}
%     A
%     = \sqrt{\frac{4D}{L^2\pi}},
%     %=\sqrt{\frac{4}{\pi}},
%     \qquad p= \frac{1}{2},
%     \qquad C = \frac{L^2}{4D}. % =\frac{1}{4}
% \end{align*}

%%%%%%%%%%%%%%%%%%%%%%%%%%%%%%%%%%%%%%%%%%%%%%%%%%%%%%%%%%%%%%%%%%%%%%%%%%%%%%%%%%%%%%%%%%%%%%%%%%%%%%%%%%%%%%%%%%%%%%%%
\subsection{Diffusive escape in three space dimensions}
The second example is the time for diffusion processes to escape a three-dimensional sphere centered on their starting location. The survival probability of $\tau$ is \cite{carslaw1959}
$$
S(t) = 1-\sqrt{\frac{4L^2}{\pi D t}}\sum_{j=0}^\infty \exp\left[-\frac{L^2}{4Dt}(2j+1)^2\right],
$$
where $L>0$ is the radius of the sphere and $D>0$ is the searcher diffusivity. Hence,
$$
1-S(t) \sim At^pe^{-C/t}, \qquad A = \sqrt{\frac{4L^2}{D\pi}}, \qquad p= -\frac{1}{2},\qquad C = \frac{L^2}{4D}.
$$
% where
% $$A
% = \sqrt{\frac{4L^2}{D\pi}},
% %=\sqrt{\frac{4}{\pi}},
% \qquad p= -\frac{1}{2},
% \qquad C = \frac{L^2}{4D}. % =\frac{1}{4}.
% $$

%%%%%%%%%%%%%%%%%%%%%%%%%%%%%%%%%%%%%%%%%%%%%%%%%%%%%%%%%%%%%%%%%%%%%%%%%%%%%%%%%%%%%%%%%%%%%%%%%%%%%%%%%%%%%%%%%%%%%%%%
\subsection{Search on a discrete network}
The third example takes place on a $5 \times 5$ discrete square grid with rate $1$ movement between neighboring vertices. The initial location is the upper left vertex of the grid as $(0, 0)$ and the target location $(2, 1)$ is two to the right and one down. Using well known results of Markov chains \cite{norris1998}, it is straightforward to find the survival probability as the product of elementary row vectors and the matrix exponential of an appropriately constructed rate matrix. For the short time behavior, as described in Section \ref{sec:model} above, Proposition~1 in \cite{lawley2020networks} implies 
$$
1-S(t) \sim At^p,\qquad A = \frac{1}{2}, \qquad p = 3.
$$
% where
% $$
% A = \frac{1}{2}, \qquad p = 3.
% $$

\subsection{Comparison to numerics}
For these three examples, Figure \ref{fig:numericsLimitingDist} compares the probability densities and means of the fastest FPT $T_1$ to their limits obtained in section~\ref{sec:results}.
\begin{figure}[ht]
    \centering
    \includegraphics[width=0.95\linewidth]{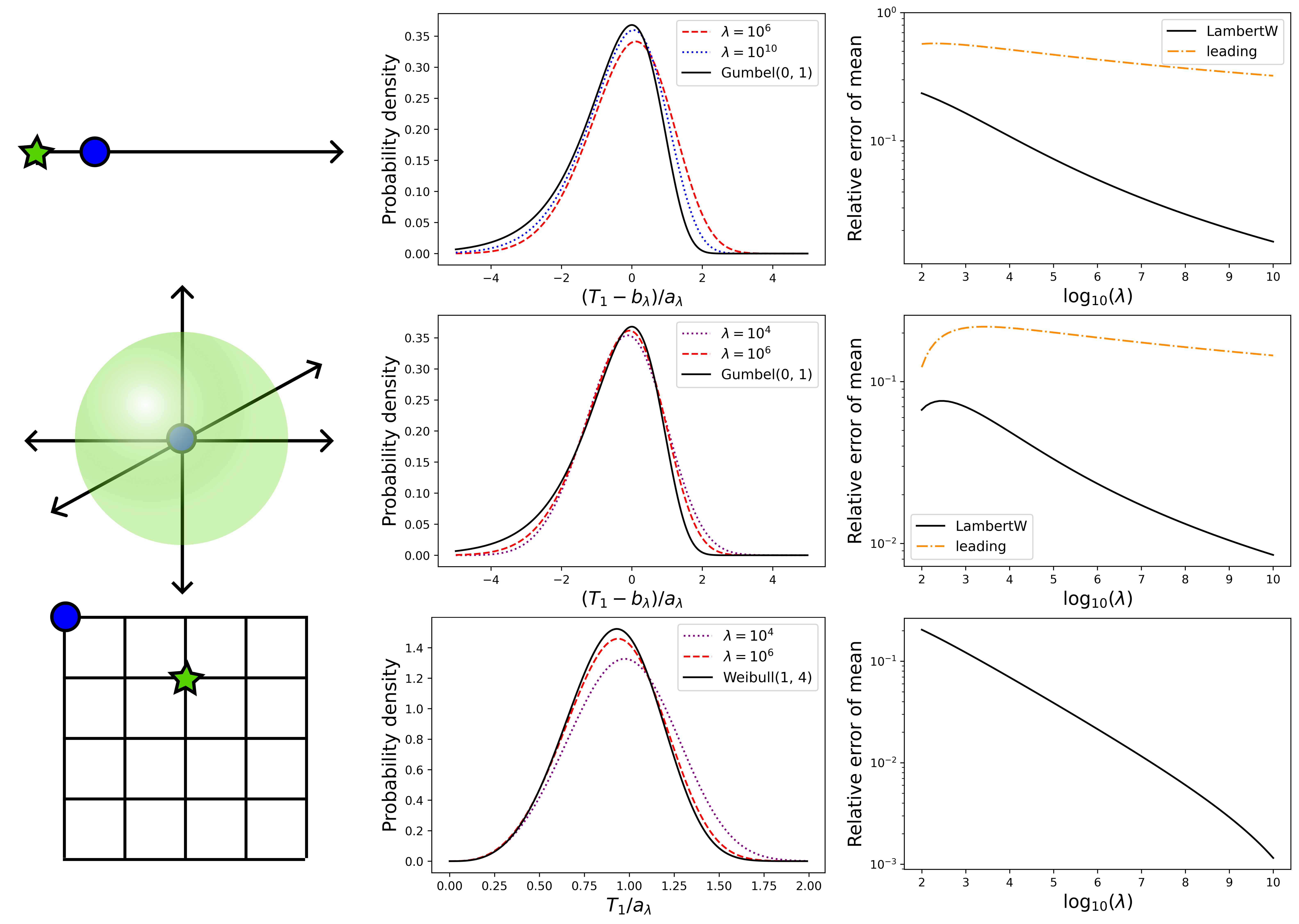}
    \caption{The first column depicts the three search examples used for numerics. Searchers start at the blue ball and their target is the green (the star, the sphere, and the star). The second column shows corresponding plots
    of the probability densities of scaled first passage time $T_1$ for large $\lambda$ as well as the theoretical limiting distribution of Weibull or Gumbel given by Theorems \ref{thm:distpower} and \ref{thm:distexp}.  The third column plots the relative error in estimating the mean fastest FPT, $\E[T_1]$, using \eqref{eq:estimatePower}, \eqref{eq:estimateExp}, or \eqref{eq:leading} as a function of immigration rate $\lambda$. We take $L=D=1$ for the diffusion examples.}
    \label{fig:numericsLimitingDist}
\end{figure}
If $1-S(t) \sim At^p$, then we use Theorem \ref{thm:momentpower} to estimate $\E[T_1]$ as
\begin{equation}
    \widehat{T_1} := \left(\frac{A\lambda}{p+1}\right)^{-1/(p+1)}\Gamma\left(1+\frac{1}{p+1}\right).
    \label{eq:estimatePower}
\end{equation}
Similarly, when $1-S(t) \sim At^pe^{-C/t}$ we use Theorem \ref{thm:momentexp} to estimate $\E[T_1]$ as
\begin{equation}
    \widehat{T_1} := b_\lambda-a_\lambda \gamma,
    \label{eq:estimateExp}
\end{equation}
where $\gamma$ is the Euler-Mascheroni constant and $a_\lambda$ and $b_\lambda$ involve the LambertW function as in Corollary \ref{coro:lamWparamOfABlam}. We also consider the leading order approximation,
\begin{equation}
    \widehat{T_1} := C/\ln(C\lambda).%-a_\lambda \gamma
    \label{eq:leading}
\end{equation}
The relative error $1-\widehat{T_1}/\E[T_1]$ of using these different mean FPT estimates is shown in the last column of Figure \ref{fig:numericsLimitingDist}.

%%%%%%%%%%%%%%%%%%%%%%%%%%%%%%%%%%%%%%%%%%%%%%%%%%%%%%%%%%%%%%%%%%%%%%%%%%%%%%%%%%%%%%%%%%%%%%%%%%%%%%%%%%%%%%%%%%%%%%%%%%%%%%%%%%%%%%%%%%%%%%%%%%%%%%%%%%%%%%%%%%%%%%%%%%%%%%%%%%%%%%%%%%%%%%%%%%%%%%%%%%%%%%%%%%%%%%%%%%%%%%%%%%%%%%%%%%%%%%%%%%%%%%%%%%%%%%%%%%%%%%%%%%%%%%%%%%%%%%%%%%%%%%%%%%%%%%%%%%%%%%%%%%%%%%%%%%%%%%%%%%%%%%%%%%%%%%%%%%%%%%%%%%%%%%%%%%%%%%%%
\section{Discussion}\label{sec:discussion}
In this paper, we examined the $k$-th fastest FPTs of a process where searchers immigrate into a system at rate $\lambda$ and move until they hit the target. We considered searchers with one of two survival probability short-time behaviors, which covers very general classes of search problems spanning from diffusion in continuous space to random walks on graphs. We first rederived an equation relating the survival probability $S(t)$ of a single searcher to the survival probability $S_I(t)$ of the process process with immigration \cite{campos2024}. Then, we found that when the immigration rate $\lambda$ is large, the limiting distributions for the survival probability $S_I$ are Gumbel or Weibull and we found the limit of the moments of the $k$-th fastest FPT based on these limiting distributions. Lastly, we performed numerical computations to illustrate our results in some canonical examples and demonstrate the convergence.

Mathematically, our analysis amounted to determining the extreme values of an infinite sequence of non-independent and non-identically distributed random variables. Our results constitute a rather rare case in which extreme value statistics have been determined analytically for such random variables. Indeed, a recent review \cite{majumdar2020} of extreme value theory notes that ``very few exact results'' are known for extreme value statistics of strongly correlated random variables, such as those that we considered.

One natural question is how our results on search with immigration compare with previous work on FPTs where there is no immigration but instead a large initial number of searchers \cite{weiss1983, lawley2020dist}. If there are $N\gg1$ initial iid searchers, each with survival probability $S_0(t)$, then the survival probability of the fastest searcher is
\begin{align*}%\label{eq:S0N}
    (S_0(t))^N = \exp\left[N \ln(S_0(t))\right] \approx \exp[-N(1-S_0(t))],
\end{align*}
where the approximation is due to Taylor expanding $\ln(S(t))$ at small time $t$, which is appropriate since $(S_0(t))^N$ vanishes exponentially fast as $N\to\infty$ outside of a neighborhood of $t=0$. In order to find an equivalent search with immigration process, suppose searchers with survival probability $S(t)$ immigrate at rate $\lambda$ and recall \eqref{eqn:SIequality} to obtain the requirement that
\begin{align*}
    (S_0(t))^N
    \approx \exp[-N(1-S_0(t))]
    \approx S(t) \exp\bigg[-\lambda\int_0^t(1-S(s))ds\bigg].
\end{align*}
Hence, the fastest FPTs for (i) $N\gg1$ initial searchers with survival probability $S_0(t)$ and (ii) searchers with survival probability $S(t)$ which immigrate at rate $\lambda$ are roughly equivalent if
\begin{align*}
    N(1-S_0(t))
    \approx\lambda \int_0^t (1-S(s))ds\quad\text{at small time }t.
\end{align*}
Therefore, if a search with fast immigration process satisfies $1-S(t)\sim At^p$, then the equivalent many initial searcher process satisfies $1-S_0(t)\sim A_0t^{p_0}$ with $A_0=A^{1+1/p}/(p+1)$, $p_0=p+1$, and $N=\lambda A^{-1/p}\gg1$. Similarly, if $1-S(t)\sim At^pe^{-C/t}$, then the equivalent many initial searcher process satisfies $1-S_0(t)\sim A_0t^{p_0}e^{-C_0/t}$  with $C_0=C>0$, $A_0=A/C^2$, $p_0=p+2$, and $N=\lambda C\gg1$. Though this argument is heuristic, comparing Theorem~\ref{thm:distpower} above with Theorem~2 in \cite{madrid2020comp} and Theorem~\ref{thm:distexp} above with Theorem~1 in \cite{lawley2020dist} shows that the heuristic agrees with the rigorous theory.

\appendix
\section{Proofs of theorems}

We now prove the theorems of section~\ref{sec:results}.

\subsection{Proving Theorems \ref{thm:distpower} and \ref{thm:distexp}}
\begin{proof}[Proof of Theorem~\ref{thm:distpower}]
    Let $\eps\in(0,1)$ and define $S_{\pm}(t):=1-(1\mp\eps)At^p$. By \eqref{eq:assumptionexp}, there exists a $\delta>0$ so that $S_-(t)
        \le S(t)
        \le S_+(t)$ if $t\in(0,\delta]$. Hence,
    \begin{align*}
        &S_-(t)\exp\Big[-\lambda\int_0^t(1-S_-(s))\,\dd s\Big]
        \le S_I(t)
        =S(t)\exp\Big[-\lambda\int_0^t(1-S(s))\,\dd s\Big]\\
        &\quad\le S_+(t)\exp\Big[-\lambda\int_0^t(1-S_+(s))\,\dd s\Big]\quad\text{if }t\in(0,\delta].
    \end{align*}
    A direct calculation shows that
    \begin{align*}
        &\lim_{\lambda\to\infty}S_\pm(a_\lambda x)\exp\Big[-\lambda\int_0^{a_\lambda x}(1-S_\pm(s))\,\dd s\Big]
        =\exp(-(1\mp \eps)x^{p+1})\quad\text{if }x\in\R.
    \end{align*}
    Since $\lim_{\lambda\to\infty}a_\lambda=0$ and $\eps\in(0,1)$ is arbitrary, the proof is complete.
\end{proof}

\begin{proof}[Proof of Theorem~\ref{thm:distexp}]
    Let $\eps\in(0,1)$ and define $S_{\pm}(t):=1-(1\mp\eps)At^pe^{-C/t}$. By \eqref{eq:assumptionexp}, there exists a $\delta>0$ so that $S_-(t)
        \le S(t)
        \le S_+(t)$ if $t\in(0,\delta]$. Hence,
    \begin{align*}
        &S_-(t)\exp\Big[-\lambda\int_0^t(1-S_-(s))\,\dd s\Big]
        \le S_I(t)
        =S(t)\exp\Big[-\lambda\int_0^t(1-S(s))\,\dd s\Big]\\
        &\quad\le S_+(t)\exp\Big[-\lambda\int_0^t(1-S_+(s))\,\dd s\Big]\quad\text{if }t\in(0,\delta].
    \end{align*}
    A direct calculation shows that
    \begin{align*}
        &\lim_{\lambda\to\infty}S_\pm(a_\lambda x+b_\lambda)\exp\Big[-\lambda\int_0^{a_\lambda x+b_\lambda}(1-S_\pm(s))\,\dd s\Big]
        =\exp(-(1\mp \eps)e^x)\quad\text{if }x\in\R.
    \end{align*}
    Since $\lim_{\lambda\to\infty}a_\lambda=\lim_{\lambda\to\infty}b_\lambda=0$ and $\eps\in(0,1)$ is arbitrary, the proof is complete.
\end{proof}

\begin{lemma}\label{lem:power}
    Assume
    \begin{align}\label{eq:assumptionpower2}
        F(t)
        \sim At^p\quad\text{as }t\to0+,
    \end{align}
    where $A>0$ and $p>0$. If $a_\lambda=(A\lambda/(p+1))^{-1/(p+1)}$, then for each $x\in\R$,
    \begin{align*}
        \lambda\int_0^{a_\lambda x}F(s)\,\dd s\to x^{p+1}\quad\text{as }\lambda\to\infty.
    \end{align*}
\end{lemma}

\begin{proof}[Proof of Lemma~\ref{lem:power}]
    Let $\eps\in(0,1)$ and define $F_{\pm}(t):=(1\pm\eps)At^p$. By \eqref{eq:assumptionpower2}, there exists a $\delta>0$ so that $F_-(t)
        \le F(t)
        \le F_+(t)$ if $t\in(0,\delta]$. Hence,
    \begin{align*}
        &\lambda\int_0^t F_-(s)\,\dd s
        \le \lambda\int_0^t F(s)\,\dd s
        \le\lambda\int_0^t F_+(s)\,\dd s\quad\text{if }t\in(0,\delta].
    \end{align*}
    Now, 
    \begin{align*}
    \lambda\int_0^{a_\lambda x} F_{\pm}(s)\,\dd s
        &=(1\pm\eps)x^{p+1}\quad\text{if }x\in\R.
    \end{align*}
    Since $\lim_{\lambda\to\infty}a_\lambda=0$ and $\eps\in(0,1)$ is arbitrary, the proof is complete.
\end{proof}

\begin{lemma}\label{lem:exp}
    Assume
    \begin{align}\label{eq:assumptionexp2}
        F(t)
        \sim At^pe^{-C/t}\quad\text{as }t\to0+,
    \end{align}
    where $C>0$, $A>0$, and $p\in\R$. If
    \begin{align*}
        a_\lambda
        =\frac{C}{(\ln (C\lambda))^2},\quad
        b_\lambda
        =\frac{C}{\ln (C\lambda)}
        +\frac{C(p+2)\ln(\ln(C\lambda))}{(\ln (C\lambda))^2}
        -\frac{C\ln(AC^p)}{(\ln (C\lambda))^2},
    \end{align*}
    then for each $x\in\R$,
    \begin{align*}
        \lambda\int_0^{a_\lambda x+b_\lambda}F(s)\,\dd s\to e^x\quad\text{as }\lambda\to\infty.
    \end{align*}
\end{lemma}

\begin{proof}[Proof of Lemma~\ref{lem:exp}]
    Let $\eps\in(0,1)$ and define $F_{\pm}(t):=(1\pm\eps)At^pe^{-C/t}$. By \eqref{eq:assumptionexp2}, there exists a $\delta>0$ so that $F_-(t)
        \le F(t)
        \le F_+(t)$ if $t\in(0,\delta]$. Hence,
    \begin{align*}
        &\lambda\int_0^t F_-(s)\,\dd s
        \le \lambda\int_0^t F(s)\,\dd s
        \le\lambda\int_0^t F_+(s)\,\dd s\quad\text{if }t\in(0,\delta].
    \end{align*}
    Now, 
    \begin{align*}
        \lambda\int_0^t F_{\pm}(s)\,\dd s
        =\lambda (1\pm\eps) A C^{p+1} \Gamma (-p-1,C/t),
    \end{align*}
    where $\Gamma(r,z)=\int_z^\infty u^{r-1}e^{-u}\,\dd u$ is the upper incomplete gamma function. Hence, it follows from the asymptotics, $\Gamma(r,z)
        \sim z^{r-1}e^{-z}$ as $z\to\infty$, that
    \begin{align*}
    \lim_{\lambda\to\infty}\lambda\int_0^{a_\lambda x+b_\lambda} F_{\pm}(s)\,\dd s
        &=(1\pm\eps)A C^{p+1}\lim_{\lambda\to\infty}\lambda  \Gamma (-p-1,C/(a_\lambda x+b_\lambda))\\
        &=(1\pm\eps)e^x\quad\text{if }x\in\R.
    \end{align*}
    Since $\lim_{\lambda\to\infty}a_\lambda=\lim_{\lambda\to\infty}b_\lambda=0$ and $\eps\in(0,1)$ is arbitrary, the proof is complete.
\end{proof}

\subsection{Proving Theorems \ref{thm:momentpower} and \ref{thm:momentexp}}
Let $T_1$ be the FPT in \eqref{eq:T1} when the immigration rate is $\lambda$. Further, let
$$
\widetilde{T_1} = \frac{T_1-b_{\lambda}}{a_\lambda},
$$
where $a_\lambda$ and $b_\lambda$ copied from Theorems \ref{thm:distpower} and \ref{thm:distexp} depending on the short-time behavior of $1-S(t)$ (we set $b_\lambda = 0$ if $1-S(t)\sim At^p$).

%$$\begin{array}{|c|c|c|}
%\hline
 %1-S(t)\sim        & a_\lambda & b_\lambda \\\hline
 %At^p           & (A\lambda/(p+1))^{-1/(p+1)} & 0 \\\hline
 %At^pe^{-C/t}   & \frac{C}{(\ln (C\lambda))^2} & \frac{C}{\ln (C\lambda)}
  %      +\frac{C(p+2)\ln(\ln(C\lambda))}{(\ln (C\lambda))^2}
   %     -\frac{C\ln(AC^p)}{(\ln (C\lambda))^2}\\
%\hline
%\end{array}$$

The goal of this appendix is to show that the moments of $\widetilde{T_1}$ converge to the moments of the limiting distribution. It is well known that this can be achieved by showing that for any positive integer $m$, the set of random variables $\{\widetilde{T_1}^m\}_{\lambda>\lambda^*}$ is uniformly integrable (see, for example, Theorem~3.5 in \cite{billingsley2013}).

Recall that the definition of uniform integrability is
$$
\lim_{K\rightarrow \infty} \sup_{\lambda>\lambda^*} \E[|\widetilde{T_1}|^m \mathbbm{1}_{\{|\widetilde{T_1}|^m > K^m\}}] = 0.
$$
First, we bound the expected value for some $K>0$. Noting that $|\widetilde{T_1}|^m \mathbbm{1}_{\{|\widetilde{T_1}|^m > K^m\}}$ is a nonnegative random variable, its $m$th moment is given by
\begin{align*}
    \E[|\widetilde{T_1}|^m \mathbbm{1}_{\{|\widetilde{T_1}|^m > K^m\}}] &= \int_0^{\infty} mx^{m-1}P(|\widetilde{T_1}| \mathbbm{1}_{\{|\widetilde{T_1}|^m > K^m\}} > x) dx \\   
    &= \int_K^{\infty} mx^{m-1}P(|\widetilde{T_1}| > x) dx \\ 
    &= \int_K^{\infty} mx^{m-1}(P(\widetilde{T_1} > x) + 1 -P(\widetilde{T_1} > -x)) dx.
\end{align*}
Noting that
\begin{align*}
    P(\widetilde{T_1} > x)
    &= P(T_1 > a_{\lambda} x + b_{\lambda})
    = S_I(a_{\lambda} x + b_{\lambda}),\\
    P(\widetilde{T_1} > -x)
    &= P(T_1 > -a_{\lambda} x + b_{\lambda})
    = S_I(-a_{\lambda} x + b_{\lambda}),
\end{align*}
we can decompose $\E[|\widetilde{T_1}|^m \mathbbm{1}_{\{|\widetilde{T_1}|^m > K^m\}}]$ into the sum of the following two terms,
\begin{align}
    &\int_K^{\infty} mx^{m-1} S_I(a_{\lambda} x + b_{\lambda}) dx 
    \label{eqn:limitingproofint1},\\
    &\int_K^{\infty} mx^{m-1} (1-S_I(-a_{\lambda} x + b_{\lambda})) dx. 
    \label{eqn:limitingproofint2}
\end{align}

\begin{lemma}
    For either short time decay behavior of the survival probability,
    $$\lim_{K\rightarrow \infty}\sup_{\lambda>\lambda^*}\int_K^{\infty} mx^{m-1} S_I(a_{\lambda} x + b_{\lambda}) dx=0. $$
    \label{lem:limitingproofint1}
\end{lemma}
\begin{proof}
    Plugging in the equation for $S_I$ in \eqref{eqn:SIequality} yields
\begin{align*}
    \int_K^{\infty} mx^{m-1} S_I(a_{\lambda} x + b_{\lambda}) dx &= \int_K^{\infty} mx^{m-1} S(a_{\lambda} x + b_{\lambda})\exp\left[-\lambda \int_0^{a_{\lambda} x + b_{\lambda}}1-S(s)ds\right] dx \\
    &\leq \int_K^{\infty} mx^{m-1}\exp\left[-\lambda \int_0^{a_{\lambda} x + b_{\lambda}}1-S(s)ds\right] dx \\
    &\leq \int_K^{\infty} mx^{m-1}\exp\left[-\lambda \int_{a_\lambda + b_\lambda}^{a_{\lambda} x + b_{\lambda}}1-S(s)ds\right] dx \\
    &\leq \int_K^{\infty} mx^{m-1}e^{-\lambda a_\lambda (x-1)(1-S(a_\lambda + b_\lambda))} dx.
\end{align*}
The first inequality follows from the bound $S(s)\le1$ since $S$ is a probability. The second follows from $1-S(s)\geq 0$. The third follows from lower bounding the integral by multiplying the length of the interval of integration with the minimum of $1-S(s)$ on the interval, which is when $s=a_\lambda + b_{\lambda}$ by the monotonicity of $S(s)$. As such, it suffices to show that 
$$
\lim_{K\rightarrow \infty}\sup_{\lambda>\lambda^*}\int_K^{\infty} mx^{m-1}e^{-A_\lambda (x-1)} dx=0,
$$
where $A_\lambda = \lambda a_\lambda (1-S(a_\lambda + b_\lambda))$. To do so, we note that when $1-S(t) \sim At^p$ or $1-S(t) \sim At^pe^{-C/t}$,
$$
\lim_{\lambda\rightarrow \infty} A_\lambda= p+1 \text{ or } \lim_{\lambda\rightarrow \infty} A_\lambda=e,
$$
respectively. Recall that $p>0$ when the short term behavior is $1-S(t)\sim At^p$ and therefore both limits are positive. We then pick $\lambda^*$ so that $A_\lambda\ge1/2$ for all $\lambda>\lambda^*$ and define $A^* := \inf_{\lambda > \lambda^*} A_\lambda > 0$. As such,
$$
\lim_{K\rightarrow \infty}\sup_{\lambda>\lambda^*}\int_K^{\infty} mx^{m-1}e^{-A_\lambda (x-1)} dx \leq \lim_{K\rightarrow \infty}\int_K^{\infty} mx^{m-1}e^{-A^* (x-1)} dx = 0.
$$
\end{proof}

\begin{lemma}
    For either short time decay behavior of the survival probability,
    \begin{align}
    \lim_{K\rightarrow \infty}\sup_{\lambda>\lambda^*}\int_K^{\infty} mx^{m-1} (1-S_I(-a_{\lambda} x + b_{\lambda})) dx=0.\label{eq:lemmaint}
        \end{align}\label{lem:limitingproofint2}
\end{lemma}
\begin{proof}
    We note this is trivial when $1-S(t) \sim At^p$ since under this case $b_\lambda = 0$, which implies $S_I(-a_\lambda x + b_\lambda) = 1$, making the integral $0$ regardless of the choice of $K$ or $\lambda$. As such, we only need to consider the case where $1-S(t) \sim At^pe^{-C/t}$.

    Noting that
    \begin{align*}
        1-S(t)e^{-\lambda \int_0^{t}1-S(s)ds} &= 1-e^{-\lambda \int_0^{t}1-S(s)ds} + (1-S(t))e^{-\lambda \int_0^{t}1-S(s)ds} \\
        &\leq  1-e^{-\lambda \int_0^{t}1-S(s)ds} + (1-S(t)),
    \end{align*}
    and that the integral in \eqref{eq:lemmaint} is nonzero when $x<b_\lambda/a_\lambda$, we can split the problem into showing the following two limits,
    \begin{align}
        \lim_{K\rightarrow \infty}\sup_{\lambda>\lambda^*}\int_K^{b_\lambda/a_\lambda} mx^{m-1} (1-S(-a_{\lambda} x + b_{\lambda})) dx
        &=0
        \label{eqn:lemmahelp1},\\
        \lim_{K\rightarrow \infty}\sup_{\lambda>\lambda^*}\int_K^{b_\lambda/a_\lambda} mx^{m-1} (1-e^{-\lambda \int_0^{-a_{\lambda} x + b_{\lambda}}1-S(s)ds}) dx
        &=0.
        \label{eqn:lemmahelp2}
    \end{align}
    We first show \eqref{eqn:lemmahelp1}. For any ${\eps}$ we can find a $\lambda^*$ such that
    $$1-S(-a_{\lambda} x + b_{\lambda}) \leq (1+{\eps})A(b_\lambda-a_\lambda x)^pe^{-C/(b_\lambda-a_\lambda x)}\quad\text{if }\lambda>\lambda^*.$$
    By Lemma \ref{lem:uniformconverge}, we can then find the upper bound
    $$1-S(-a_{\lambda} x + b_{\lambda}) \leq (1+{\eps})^2\frac{C}{\lambda}e^{-x} \leq (1+{\eps})^2Ce^{-x}$$
    when $\lambda^*>1$. Substituting back into \eqref{eqn:lemmahelp1}, we find
    \begin{align*}
        &\lim_{K\rightarrow \infty}\sup_{\lambda>\lambda^*}\int_K^{b_\lambda/a_\lambda} mx^{m-1} (1-S(-a_{\lambda} x + b_{\lambda})) dx\\
        &\quad\leq \lim_{K\rightarrow \infty}\sup_{\lambda>\lambda^*}\int_K^{b_\lambda/a_\lambda} mx^{m-1} (1+{\eps})^2Ce^{-x} dx \\
        &\quad\leq \lim_{K\rightarrow \infty}\int_K^{\infty} mx^{m-1} (1+{\eps})^2Ce^{-x} dx
        =0.
    \end{align*}

    We now show \eqref{eqn:lemmahelp2}. By the short time behavior of $S$,
    \begin{align*}
        &\lim_{K\rightarrow \infty}\sup_{\lambda>\lambda^*}\int_K^{b_\lambda/a_\lambda} mx^{m-1} (1-e^{-\lambda \int_0^{-a_{\lambda} x + b_{\lambda}}1-S(s)ds}) dx\\
        &\quad\leq \lim_{K\rightarrow \infty}\sup_{\lambda>\lambda^*}\int_K^{b_\lambda/a_\lambda} mx^{m-1} (1-e^{-\lambda(1+{\eps}) \frac{A}{C}(b_\lambda - a_\lambda x)^{p+2}e^{-C/(b_\lambda - a_\lambda x)}}) dx.
    \end{align*}
    It then follows by Lemma \ref{lem:uniformconverge},
    \begin{align*}
        &\lim_{K\rightarrow \infty}\sup_{\lambda>\lambda^*}\int_K^{b_\lambda/a_\lambda} mx^{m-1} (1-e^{-\lambda \int_0^{-a_{\lambda} x + b_{\lambda}}1-S(s)ds}) dx \\
        &\leq \lim_{K\rightarrow \infty}\sup_{\lambda>\lambda^*}\int_K^{\infty} mx^{m-1} (1-e^{-(1+{\eps})^2e^{-x}}) dx\\
        &= \lim_{K\rightarrow \infty}\int_K^{\infty} mx^{m-1} (1-e^{-(1+{\eps})^2e^{-x}}) dx
        =0.
    \end{align*}
\end{proof}

\begin{lemma}
    For any ${\eps}>0$, there exists a $\lambda^*$ such that %for all $\lambda > \lambda^*$ and all $x \leq b_\lambda/a_\lambda$,
    $$
    \lambda \frac{A}{C}(b_\lambda - a_\lambda x)^{p+2}e^{-C/(b_\lambda - a_\lambda x)} \leq (1+{\eps})e^{-x}\quad\text{for all }\lambda>\lambda^*\text{ and }x\le b_\lambda/a_\lambda.
    $$
    \label{lem:uniformconverge}
\end{lemma}
\begin{proof}
The definition of $a_\lambda$ and $b_\lambda$ in Theorem~\ref{thm:distexp} yields
    \begin{align*}
        &\lambda \frac{A}{C}(b_\lambda - a_\lambda x)^{p+2}e^{-C/(b_\lambda - a_\lambda x)}\\
        &\quad= \exp[\ln(C\lambda) + (p+2)\ln(b_{\lambda}-a_{\lambda}x)-C/(b_{\lambda}-a_{\lambda}x) + \ln(AC^{-2})]\\
        &\quad= \exp\left[u +p^*\ln(\frac{u+p^*\ln(u) - X}{u^2})-\frac{u^2}{u+p^*\ln(u) - X}+X-x\right],
    \end{align*}
    with the last equality following from the substitutions
    $$u = \ln(C\lambda), \quad X = x+\ln(AC^p), \quad p^* = p+2.$$
    We now focus on bounding the following part of the exponent,
    \begin{align}\label{eq:partofexponent}
        u +p^*\ln(\frac{u+p^*\ln(u) - X}{u^2})-\frac{u^2}{u+p^*\ln(u) - X}+X.
    \end{align}
    Taking the derivative of \eqref{eq:partofexponent} with respect to $X$, we find that the maximum of \eqref{eq:partofexponent} when $x\leq b_\lambda/a_\lambda$, which is equivalent to $X\leq u+p^*\ln(u)$, occurs at
    $$
    X = u+p^*\ln(u)-\frac{1}{2}\left(p^*+\sqrt{(p^*)^2+4u^2}\right).
    $$
    Plugging $X$ back in, the maximum of \eqref{eq:partofexponent} is
    \begin{align}\label{eq:max}
        2u-\sqrt{(p^*)^2+4u^2}+p^*\ln\left(\frac{p^*+\sqrt{(p^*)^2+4u^2}}{2u}\right).
    \end{align}
    Since the limit of \eqref{eq:max} as $u=\ln(C\lambda)\rightarrow \infty$ is $0$, there exists a sufficiently large $u=\ln(C\lambda)$ such that \eqref{eq:max} is always less than $\ln(1+{\eps})$. This implies that for sufficiently large $\lambda$,
    \begin{align*}
        &\lambda \frac{A}{C}(b_\lambda - a_\lambda x)^{p+2}e^{-C/(b_\lambda - a_\lambda x)}\\
        &\quad\leq \exp\left[u +p^*\ln(\frac{u+p^*\ln(u) - X}{u^2})-\frac{u^2}{u+p^*\ln(u) - X}+X-x\right] \\
        &\quad\leq \exp\left[\ln(1+{\eps})-x\right]
        = (1+{\eps})e^{-x}.
    \end{align*}
\end{proof}

\subsection{Proofs for Section \ref{sec:kth}}
We start with a proof of Lemma \ref{lem:Tksurvival}.
\begin{proof}[Proof of Lemma~\ref{lem:Tksurvival}]
    Because 
    \begin{align}\label{eq:convenient}
        P(T_k > t) = P(T_1>t) + \sum_{j=1}^{k-1} P(T_{j+1} > t > T_{j}),% =S_I(t)+ \sum_{j=1}^{k-1} P(T_{j+1} > t > T_{j}),
    \end{align}
    all we need to finish the proof is find $P(T_{j+1} > t > T_{j})$, i.e.\ the probability exactly $j$ searchers have found the target by time $t$. To find this probability, we start by conditioning on the number of particles at time $t$ (denoted $N(t)$), which allows us to draw all but the original searcher's immigration times independently from the uniform distribution. Therefore, letting $F(t)=1-S(t)$ yields
    \begin{align*}
        P(T_{j+1} > t>T_j | N(t) = n) &= F(t)\left[\frac{1}{t} \int_0^t F(s)ds\right]^{j-1}\left[\frac{1}{t} \int_0^t S(s)ds\right]^{n-j}\binom{n-1}{j-1}\\
        &\quad+S(t)\left[\frac{1}{t} \int_0^t F(s)ds\right]^{j}\left[\frac{1}{t} \int_0^t S(s)ds\right]^{n-j-1}\binom{n-1}{j},
    \end{align*}
    where the first line indicates the probability that the original searcher (i.e.\ the searcher present at time 0) and exactly $j-1$ additional searchers find the target by time $t$, and the second line indicates the probability that exactly $j$ searchers, none of whom are the original searcher, have found the target by time $t$.
    
    Using the law of total probability and that the number of searchers added in by time $t$ follows a Poisson distribution with mean $\lambda t$,
    $$
    P(T_{j+1} > t>T_j) = \sum_{n=j}^\infty P(T_{j+1} > t>T_j | N(t) = n) \frac{e^{-\lambda t}(\lambda t)^{n-1}}{(n-1)!}.
    $$
    Combining our expressions, we find that
    \begin{align}\label{eq:lemmauseful}
        P(T_{j+1} > t>T_j) = \exp\left[-\lambda \int_0^t F(s)ds\right]\left[F(t)\frac{\left(\lambda \int_0^t F(s)ds\right)^{j-1}}{(j-1)!} + S(t)\frac{\left(\lambda \int_0^t F(s)ds\right)^{j}}{j!}\right].
    \end{align}
    We are now ready to use our equation for $P(T_k>t)$. Substituting $S_I(t)$ from \eqref{eqn:SIequality} and taking advantage of the telescoping nature of the sum, it follows that
    $$P(T_k > t) = \exp\left[-\lambda \int_0^t F(s)ds\right]\left[S(t)\frac{\left(\lambda \int_0^t F(s)ds\right)^{k-1}}{(k-1)!} + \sum_{j=1}^{k-1} \frac{\left(\lambda \int_0^t F(s)ds\right)^{j-1}}{(j-1)!} \right],$$
    which is what we aimed to prove.
\end{proof}

\begin{proof}[Proof of Theorem~\ref{thm:kthdistpower}]
    Setting $t=a_\lambda x$ in \eqref{eq:lemmauseful}, taking $\lambda\to\infty$, and using Lemma~\ref{lem:power} and the fact that $\lim_{\lambda\to\infty}a_\lambda=0$ yields
    \begin{align*}
        P(T_{j+1}>a_\lambda x>T_j)
        =e^{-x^{p+1}}\frac{x^{j(p+1)}}{j!}.
    \end{align*}
    Hence, \eqref{eq:convenient} and Theorem~\ref{thm:distexp} yield
    \begin{align*}
        \lim_{\lambda\to\infty}P(T_k>a_\lambda x+b_\lambda)
        =e^{-x^{p+1}}\sum_{j=0}^{k-1}\frac{x^{j(p+1)}}{j!},
    \end{align*}
    which completes the proof.
\end{proof}

\begin{proof}[Proof of Theorem~\ref{thm:kthdistexp}]
    Setting $t=a_\lambda x+b_\lambda$ in \eqref{eq:lemmauseful}, taking $\lambda\to\infty$, and using Lemma~\ref{lem:exp} and the fact that $\lim_{\lambda\to\infty}a_\lambda=\lim_{\lambda\to\infty}b_\lambda=0$ yields
    \begin{align*}
        P(T_{j+1}>a_\lambda x+b_\lambda>T_j)
        =e^{-e^x}\frac{e^{jx}}{j!}.
    \end{align*}
    Hence, \eqref{eq:convenient} and Theorem~\ref{thm:distexp} yield
    \begin{align}\label{eq:sumexpression}
        \lim_{\lambda\to\infty}P(T_k>a_\lambda x+b_\lambda)
        =e^{-e^x}\sum_{j=0}^{k-1}\frac{e^{jx}}{j!}.
    \end{align}
    Taking the derivative of the righthand side of \eqref{eq:sumexpression} with respect to $x$ and multiplying by $-1$ yields the probability density in Theorem~\ref{thm:kthdistexp}.
\end{proof}

%\section*{Appendix: Proof of Theorem \ref{thm:kthmoment}: Moments for $k$-th fastest}

\begin{proof}[Proof of Theorem~\ref{thm:kthmoment}]
We once again show uniform integrability for $\{(T_k-b_\lambda)/a_\lambda\}_{\lambda > \lambda^*}$ to show that the moments converge to the moments of the limiting distribution. Using for the work for the moments of first fastest (i.e.\ $k=1$), we replace \eqref{eqn:limitingproofint1} and \eqref{eqn:limitingproofint2} with needing to show the following two equalities,
\begin{align}
    \lim_{K\rightarrow \infty}\sup_{\lambda>\lambda^*}\int_K^{\infty} mx^{m-1} P(T_k>a_{\lambda} x + b_{\lambda}) dx =0
    \label{eqn:limitingproofkthint1},\\
    \lim_{K\rightarrow \infty}\sup_{\lambda>\lambda^*}\int_K^{\infty} mx^{m-1} (1-P(T_k>-a_{\lambda} x + b_{\lambda})) dx =0
    \label{eqn:limitingproofkthint2}.
\end{align}
We start with the easier equation, \eqref{eqn:limitingproofkthint2}. Note that since the $k$-th fastest comes after the first fastest, 
$$
P(T_k>-a_{\lambda} x + b_{\lambda}) \geq P(T_1>-a_{\lambda} x + b_{\lambda}) = S_I(-a_{\lambda} x + b_{\lambda}),
$$ 
and therefore Lemma~\ref{lem:limitingproofint2} implies
\begin{align*}
&\lim_{K\rightarrow \infty}\sup_{\lambda>\lambda^*}\int_K^{\infty} mx^{m-1} (1-P(T_k>-a_{\lambda} x + b_{\lambda})) dx \\
&\quad\leq \lim_{K\rightarrow \infty}\sup_{\lambda>\lambda^*}\int_K^{\infty} mx^{m-1} (1-S_I(-a_{\lambda} x + b_{\lambda})) dx = 0.
\end{align*}

Now, we examine \eqref{eqn:limitingproofkthint1}. Looking back at Theorem \ref{lem:Tksurvival}, we see that
$$
P(T_k > t) \leq \exp\left[-\lambda \int_0^t F(s)ds\right]\left[ \sum_{j=1}^{k} \frac{\left(\lambda \int_0^t F(s)ds\right)^{j-1}}{(j-1)!} \right]
$$
where $F(t)=1-S(t)$, and therefore it suffices to show that
\begin{equation}
    \lim_{K\rightarrow \infty}\sup_{\lambda>\lambda^*}\int_K^{\infty} mx^{m-1} \exp\left[-\lambda \int_0^{a_{\lambda} x + b_{\lambda}} F(s)ds\right]\left[ \frac{\left(\lambda \int_0^{a_{\lambda} x + b_{\lambda}} F(s)ds\right)^{j}}{j!} \right] dx =0,
\end{equation}
for any integer $j \geq 0$. As the $j=0$ case was done earlier (see proof of Lemma \ref{lem:limitingproofint1}), we focus on $j\geq 1$. For ease of notation, define
$$
G := \lambda \int_0^{a_{\lambda} x + b_{\lambda}} F(s)ds.
$$
Then our integral of interest becomes
\begin{align*}
    \int_K^{\infty} mx^{m-1} \exp\left[-G\right] \frac{G^{j}}{j!}  dx &\leq \int_K^{\infty} mx^{m-1} \exp\left[-G\right] G^{j}  dx \\
    %&= \int_{\max(K, x_1)}^{\infty} mx^{m-1} \exp\left[-G\right] G^{j}  dx + \int_K^{\max(K, x_1)} mx^{m-1} \exp\left[-G\right] G^{j}  dx
    %&=\int_K^{\infty} mx^{m-1} e^{-G +j\ln[G]} dx \\
    % &=\int_K^{\infty} mx^{m-1} e^{-G +j\ln[G]} \mathbbm{1}_{\ln[G] \leq \frac{G}{2j}} (x) dx \\
    % &\phantom{=}+\int_K^{\infty} mx^{m-1} e^{-G +j\ln[G]} \mathbbm{1}_{\ln[G] > \frac{G}{2j}} (x) dx \\
    &=\int_K^{\infty} mx^{m-1} e^{-G}G^j \mathbbm{1}_{\ln[G] \leq \frac{G}{2j}} (x) dx \\
    &\phantom{=}+\int_K^{\infty} mx^{m-1} e^{-G}G^j \mathbbm{1}_{\ln[G] > \frac{G}{2j}} (x) dx.
\end{align*}
%where $x_1$ is the largest $x$ value where $\ln[G] = \frac{G}{2j}$ or, if there is no such $x$, then $x_1=0$. 
We deal with the two integrals separately. For the first, note that
\begin{align*}
    \int_K^{\infty} mx^{m-1} e^{-G}G^j \mathbbm{1}_{\ln[G] \leq \frac{G}{2j}} (x) dx &= \int_K^{\infty} mx^{m-1} e^{-G + j \ln(G)} \mathbbm{1}_{\ln[G] \leq \frac{G}{2j}} (x) dx \\ 
    &\leq \int_K^{\infty} mx^{m-1} e^{-G/2} \mathbbm{1}_{\ln[G] \leq \frac{G}{2j}} (x) dx \\
    &\leq \int_K^{\infty} mx^{m-1} e^{-G/2} dx.
\end{align*}
From here, it is straightforward to repeat the steps of the proof of Lemma \ref{lem:limitingproofint1} to show that 
$$
\lim_{K\rightarrow \infty}\sup_{\lambda>\lambda^*} \int_K^{\infty} mx^{m-1} e^{-G}G^j \mathbbm{1}_{\ln[G] \leq \frac{G}{2j}} (x) dx \leq \lim_{K\rightarrow \infty}\sup_{\lambda>\lambda^*} \int_K^{\infty} mx^{m-1} e^{-G/2} dx = 0.
$$
Now for the second integral, we aim to show
$$
\lim_{K\rightarrow \infty}\sup_{\lambda>\lambda^*} \int_K^{\infty} mx^{m-1} e^{-G}G^j \mathbbm{1}_{\ln[G] > \frac{G}{2j}} (x) dx = 0.
$$
This immediately holds when $j=1$ ($\ln(x)<x/2$ for all reals $x$), so we focus on when $j\geq 2$. Let $G^*$ be the largest value of $G$ where $\ln[G] = \frac{G}{2j}$. Then, examining the $x$ value $x^*$ that gives $G^*$, 
$$G^* = \lambda \int_0^{a_{\lambda} x^* + b_{\lambda}} F(s)ds \geq A_\lambda (x^*-1)\geq A^* (x^*-1),$$
where $A_\lambda = \lambda a_\lambda (1-S(a_\lambda + b_\lambda))=\lambda a_\lambda F(a_\lambda + b_\lambda)$. 
Recalling that $A^*=\inf_{\lambda>\lambda^*}A_\lambda>0$, this implies that
$$
x^* \leq \frac{G^*}{A^*}+1,
$$
and therefore so long as $K > \frac{G^*}{A^*}+1$, the integral is $0$ and we have uniform integrability.

All that remains is to find the moments. In the case where $1-S \sim At^pe^{-C/t}$, we find the moment generating function by definition, % then plugging into Mathematica 
$$
\int_{-\infty}^\infty e^{tx}f_{Z_k}(x) dx = \frac{1}{(k-1)!}\Gamma\left(k+t\right)
$$
In the case $1-S \sim At^p$, we retrace the derivation of the moment generating function of the Weibull. Taking a derivative of $P(Y_k<k)$ to find
$$
f_{Y_k}(x) = \frac{p+1}{(k-1)!}x^{p+(p+1)(k-1)}e^{-x^{p+1}},
$$
we then note the moment generating function is
\begin{align*}
    \int_0^\infty e^{tx}f_{Y_k}(x) dx &=\int_0^\infty e^{tx}\frac{p+1}{(k-1)!}x^{p+(p+1)(k-1)}e^{-x^{p+1}} dx \\
    &=\int_0^\infty \sum_{n=0}^{\infty} \frac{t^n}{n!}x^n\frac{p+1}{(k-1)!}x^{p+(p+1)(k-1)}e^{-x^{p+1}} dx \\
    &= \sum_{n=0}^{\infty} \frac{t^n}{n!} \int_0^\infty x^n\frac{p+1}{(k-1)!}x^{p+(p+1)(k-1)}e^{-x^{p+1}} dx \\
    &= \sum_{n=0}^{\infty} \frac{t^n}{n!} \frac{1}{(k-1)!}\Gamma\left(k+\frac{n}{p+1}\right)
\end{align*}
Taking the $m$-th derivative at zero gives the desired result.
\end{proof}

%%%%%%%%%%%%%%%%%%%%%%%%%%%%%%%%%%%%%%%%%%%%%%%%%%%%%%%%%%%%%%%%%%%%%%%%%%%%%%%%%%%%%%%%%%%%%%%%%%%%%%%%%%%%%%%%%%%%%%%%%%%%%%%%%%%%%%%%%%%%
\subsubsection*{Acknowledgments}
The authors were supported by the National Science Foundation (Grant Nos.\ CAREER DMS-1944574 and DMS-2325258).

%%%%%%%%%%%%%%%%%%%%%%%%%%%%%%%%%%%%%%%%%%%%%%%%%%%%%%%%%%%%%%%%%%%%%%%%%%%%%%%%%%%%%%%%%%%%%%%%%%%%%%%%%%%%%%%%%%%%%%%%%%%%%%%%%%%%%%%%%%%%

% Create the reference section using BibTeX:
\bibliography{library.bib}
\bibliographystyle{unsrt}

%%%%%%%%%%%%%%%%%%%%%%%%%%%%%%%%%%%%%%%%%%%%%%%%%%%%%%%%%%%%%%%%%%%%%%%%%%%%%%%%%%%%%%%%%%%%%%%%%%%%%%%%%%%%%%%%%%%%%%%%%%%%%%%%%%%%%%%%%%%%%%%%%%%%%%%%%%%%%%%%%%%%%%%%%%%%%%%%%%%%%%%%%%%%%%%%%%%%%%%%%%%%%%%%%%%%%%%%%%%%%%%%
\end{document}